\documentclass[10pt]{amsart}

\usepackage[utf8]{inputenc}
\usepackage{amsmath,amssymb,amsthm}
\usepackage{mathpazo}
\usepackage[a4paper,margin=1in]{geometry}
\usepackage{hyperref}
\usepackage{graphicx}
\newtheorem{theorem}{Theorem}[section]
\newtheorem{lemma}[theorem]{Lemma}
\theoremstyle{remark}
\newtheorem{remark}[theorem]{Remark}

\newcommand{\ssp}{\operatorname{ssp}}

\title[Strong path separation]
{Strongly separating graph edges with $10n$ paths}

\author{Xiao-Chuan Liu}
\address[Liu]{Departamento de Matemática,
 Universidade Federal de Pernambuco,
	Avenida Jornalista Aníbal Fernandes, Cidade Universitária, Recife, 50740-540, Brasil}
\email{xiaochuan.liu@ufpe.br}
\author{Boyan Xu}
\address[Xu]{School of Data Science and Information Engineering, Guizhou Minzu University, Guiyang, Guizhou Province, 550025, China}
\email{boyan04518@gmail.com}
\author{Xu Yang}
\address[Yang]{Instituto de Computação,  Universidade Federal de Alagoas,
	Av. Lourival Melo Mota, S/N, Maceió, 57072-900, Brasil}
\email{yang@ic.ufal.br}

\begin{document}

\begin{abstract}
A family of paths strongly separates the edges of a graph if every two distinct edges are separated in both directions by paths in the family. Bonamy, Botler, Dross, Naia, and Skokan proved that every $n$-vertex graph admits such a family of at most $19n$ paths. We improve this bound to $10n-o(n)$. 
\end{abstract}

\maketitle

\section{Introduction}

Separating systems are a classical topic in extremal set theory, going back to
R\'enyi~\cite{Renyi1961}.  A graph-constrained version arises when the ground
set is $E(G)$ and the members of the separating family are required to be
paths of $G$.  This restriction is natural in network design: if the edges of
$G$ represent communication links and at most one link is defective, then a
predetermined path may be used as a test, and a collection of such tests can
identify the defective link precisely when the corresponding paths separate
the edges. Thus, in contrast with unrestricted separating
systems, the tests here must respect the geometry of the underlying network.

A family $\mathcal P$ of simple paths in a graph $G$ \emph{strongly separates}
the edges of $G$ if, for every two distinct edges $e$ and $f$, there is a path
containing $e$ and avoiding $f$, and there is another path containing $f$ and
avoiding $e$.  The minimum size of such a family is the \emph{strong
separation number} $\ssp(G)$.  The requirement in both directions distinguishes
strong separation from the usual, or weak, separation condition, where one
only asks for a path containing exactly one of $e$ and $f$.

Falgas-Ravry, Kittipassorn, Kor\'andi, Letzter, and Narayanan~\cite{FKKLN2014} systematically
studied weakly separating path systems and conjectured that every $n$-vertex
graph admits one of size $O(n)$; they proved the conjecture for, among other
classes, random graphs and graphs of linear minimum degree, and obtained tight
bounds for trees.  Independently, Balogh, Csaba, Martin, and
Pluh\'ar~\cite{Balogh2016} formulated the corresponding conjecture for strong separation and
proved the general bound $O(n\log n)$.  Letzter~\cite{Letzter2024} subsequently
improved this to $O(n\log^* n)$.  The linear conjecture was
finally confirmed by Bonamy, Botler, Dross, Naia, and Skokan~\cite{BBDNS}, who proved $\ssp(G)\le 19n$
for every $n$-vertex graph $G$.

Since the existence of a linear bound was settled, a central issue has been to
determine the correct linear constant.  Recent results show that for several
important graph classes the answer is much closer to $n$ than the general
bound $19n$ suggests.  Fernandes, Mota, and Sanhueza-Matamala ~\cite{FernandesMotaSanhueza2025} proved $\ssp(K_n)=(1+o(1))n$
and, more generally, determined the asymptotic strong separation number for
certain robustly connected $\alpha n$-regular graphs.
For complete graphs, Kontogeorgiou and Stein~\cite{KontogeorgiouStein2026} subsequently obtained the explicit
upper bound $\ssp(K_n)\le n+9$.  In another
recent development, Fernandes, Hoppen, Kontogeorgiou, Mota, and Peng~\cite{FHKMP} proved
that every $2$-degenerate $n$-vertex graph satisfies $\ssp(G)\le n$; they also
derived bounds for subcubic, planar, and planar bipartite graphs, including
$\ssp(G)\le 2n$ for planar graphs and $\ssp(G)\le 3n/2$ for planar bipartite
graphs, and obtained sharp results in parts of the complete bipartite
regime.

These results make the general problem particularly compelling. On the one
hand, the examples and special classes above indicate that constants close to
\(1\) are often possible. On the other hand, a graph consisting of \(n/4\)
disjoint copies of \(K_4\) requires \(5n/4\) paths, so no universal bound of
the form \((1+o(1))n\) can hold for arbitrary, possibly disconnected graphs.
For connected graphs, however, such an asymptotic bound remains a possibility.
It is also worth noting that the earlier claimed lower bound close to \(2n\)
in~\cite{Balogh2016} relies on an incorrect value for the length of a longest
path in the relevant complete bipartite graph, as recently pointed out
in~\cite{FHKMP}. Thus a large gap remains between the known lower-bound
phenomena and the best universal upper bound, and improving the coefficient
in the general linear theorem is a natural step toward understanding the
extremal behavior of \(\ssp(G)\).

Our main result nearly halves the previously known universal coefficient.

\begin{theorem}\label{thm:main}
Let $G$ be a graph on $n$ vertices.  If $E(G)\ne\varnothing$, then
\begin{equation}\label{eq:main}
 \ssp(G)\leq 10n-3.
\end{equation}
\end{theorem}

To the best of our knowledge, Theorem~\ref{thm:main} is the first improvement
of the universal linear coefficient since the $19n$ theorem of Bonamy, Botler, Dross, Naia and Skokan~\cite{BBDNS}; the recent advances described above concern special graph
classes.  While the coefficient $10$ is certainly not expected to be optimal,
the result substantially narrows the general upper bound and introduces a
mechanism that combines path geometry with sparse auxiliary structure.
A more careful count in the same argument yields a further saving of order $\log n$, see Remark~\ref{rem:log-refinement}.

The subject has also developed in several related directions.  Weak separating
systems in complete graphs were studied by Wickes~\cite{Wickes2024}; Botler and Naia~\cite{BotlerNaia2025}
considered separation by cycles and by subdivisions of $K_4$;
Biniaz, Bose, De Carufel, Maheshwari, Miraftab, Odak, Smid, Smorodinsky and Yuditsky~\cite{BiniazEtAl2025} studied vertex-separating path and tree systems in several graph
classes; and Lichev and Sanhueza-Matamala~\cite{LichevSanhueza2026} recently determined
asymptotic results for vertex-separating path systems in random graphs.  Together with the recent work on strong edge separation,
these results illustrate a broader theme: imposing graph-theoretic structure
on the tests in a separating system creates extremal problems whose behavior
is governed simultaneously by information-theoretic separation and the
geometry of paths in the host graph.

The proof has two parts.  First, we place the vertices of a graph on a fixed
path and use edges of this path to join selected edges.  This gives a strongly
separating family of at most $5r/2$ paths for a graph on $r$ vertices placed on
the fixed path.  We then apply this result to the graphs produced by repeated
P\'osa rotations~\cite{Posa1976}.  The edges of the fixed paths form a
$2$-degenerate graph, to which we apply the theorem of Fernandes, Hoppen,
Kontogeorgiou, Mota, and Peng~\cite{FHKMP} giving a strongly separating family of at most
$n$ paths.

\section{Proof of the main result}\label{sec:proof}

\subsection{Paths along a fixed path}\label{sec:fixed-path}

Throughout this section, $Q$ is a fixed simple path.  Choose one endpoint
of $Q$ as its first endpoint.  This gives a linear order on $V(Q)$: we
write $u<v$ if $u$ occurs before $v$ when we follow $Q$ from its first
endpoint to its other endpoint.  When we move from $u$ to $v$ on $Q$, we
mean that we use the unique subpath of $Q$ between these two vertices.

The other edges considered in this section have both endpoints on $Q$,
but they are not edges of $Q$.  We use subpaths of $Q$ to join these
additional edges.  Thus every path constructed below lies in the union of
$Q$ and the stated additional edges.  The main point is to choose the
subpaths of $Q$ so that the resulting path does not repeat a vertex.

\begin{lemma}\label{lem:nested}
Suppose that
\begin{equation}\label{eq:nested-order}
 M=\{x_i y_i:1\leq i\leq t\},
 \qquad x_1<\cdots<x_t<y_t<\cdots<y_1.
\end{equation}
There is a simple path in $Q\cup M$ that contains every edge of $M$ and
has $x_1$ as an endpoint.  There is also such a path with $y_1$ as an
endpoint.
\end{lemma}

\begin{proof}
Start with $x_1y_1$.  Move backwards on $Q$ from $y_1$ to $y_2$, use
$y_2x_2$, move forwards from $x_2$ to $x_3$, and continue in this way.
The intervals of $Q$ used in this construction are disjoint.  For the
second path, start with $y_1x_1$, move forwards on $Q$ from $x_1$ to
$x_2$, use $x_2y_2$, move backwards from $y_2$ to $y_3$, and continue
in this way.
\end{proof}

We also need a version for two intervals of $Q$.  Let
$L=(\ell_0,\ldots,\ell_{d-1})$ and
$R=(r_0,\ldots,r_{d-1})$ be disjoint intervals, listed in their order on
$Q$, with every vertex of $L$ before every vertex of $R$.  The two
intervals need not be consecutive.

\begin{lemma}\label{lem:aligned}
Let $I\subseteq\{0,\ldots,d-1\}$ and
$M=\{\ell_i r_i:i\in I\}$.  If $|I|$ is odd, there is a simple path in
$Q\cup M$ that contains every edge of $M$ and goes from $\ell_0$ to
$r_{d-1}$.  There is another such path from $\ell_{d-1}$ to $r_0$.  If
$|I|$ is positive and even, there is one from $\ell_0$ to
$\ell_{d-1}$, and another from $r_0$ to $r_{d-1}$.
\end{lemma}

\begin{proof}
Write $I=\{a_1<\cdots<a_t\}$.  Starting at $\ell_0$, move on $L$ to
$\ell_{a_1}$, cross to $r_{a_1}$, move on $R$ to $r_{a_2}$, cross to
$\ell_{a_2}$, and continue.  The intervals used on $L$ and $R$ are
pairwise disjoint.  If $t$ is odd, the last added edge ends in $R$, and
the path can be extended to $r_{d-1}$.  The construction is illustrated in Figure~\ref{fig:aligned-pairs}.

If $t$ is even, it ends in $L$,
and the path can be extended to $\ell_{d-1}$.  Starting in $R$ gives the
other two paths.  For the second path in the odd case, use the indices in
decreasing order.

\begin{figure}[t]
  \centering
\includegraphics[width=0.7\textwidth]
{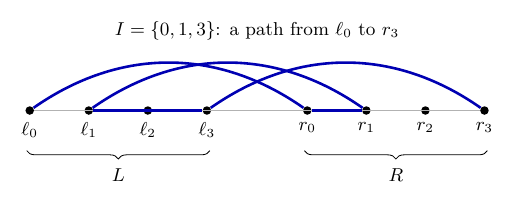}
  \caption{An illustration of Lemma~\ref{lem:aligned} for
  $I=\{0,1,3\}$. The solid path goes from $\ell_0$ to $r_3$
  and contains all pairs $\ell_i r_i$ with $i\in I$.}
  \label{fig:aligned-pairs}
\end{figure}
\end{proof}

We next regard $0,1,\ldots,q-1$ as both the vertices of $Q$ in their
linear order and the elements of $\mathbb Z_q$.  The cyclic distance
between $i$ and $j$ is $\min\{|i-j|,q-|i-j|\}$.

\begin{lemma}\label{lem:distance}
Let $q$ be odd and $1\leq d\leq(q-1)/2$.  Let $F$ be any set of pairs of
cyclic distance $d$.  There are three simple paths in $Q\cup F$ whose
union contains $F$.
\end{lemma}

\begin{proof}
Split the pairs of distance $d$ into
\begin{equation}\label{eq:two-parts}
 \mathcal A=\{\{i,i+d\}:0\leq i<q-d\},
 \qquad
 \mathcal B=\{\{i,i+q-d\}:0\leq i<d\}.
\end{equation}
For $k\geq0$, let $\mathcal A_k$ consist of the selected pairs
$\{i,i+d\}$ for which
$kd\leq i<\min\{(k+1)d,q-d\}$.  The pairs in $\mathcal A_k$ join two
intervals of length $d$, except that the last interval may be shorter.

For every $k$ such that $kd<q-d$, put
$s_k=\min\{d,q-d-kd\},$
and let
\[
 L_k=(kd,\ldots,kd+s_k-1),\qquad
 R_k=((k+1)d,\ldots,(k+1)d+s_k-1).
\]
Thus
\[
 \mathcal A_k
 =F\cap\bigl\{\{kd+a,(k+1)d+a\}:0\leq a<s_k\bigr\}.
\]
The intervals $L_k$ and $R_k$ have the same length, and $s_k=d$
except possibly for the last nonempty block.

We first treat the sets with even index.  Let $k_1<\cdots<k_m$
be the even indices for which $\mathcal A_{k_j}$ is nonempty.  If
$m=1$, Lemma~\ref{lem:aligned} gives a single path containing
$\mathcal A_{k_1}$, so suppose that $m\geq2$.  For every
$1<j<m$ such that $|\mathcal A_{k_j}|$ is even, temporarily set
aside one pair from $\mathcal A_{k_j}$.  The choices of the pairs set
aside will be specified below.  Every internal set now contains an odd
number of pairs.

For an internal index $k_j$, Lemma~\ref{lem:aligned} gives a path
through the remaining pairs of $\mathcal A_{k_j}$ from the first
vertex of $L_{k_j}$ to the last vertex of $R_{k_j}$.  Thus this local
path has one endpoint on the side facing the preceding block and the
other on the side facing the following block.

For the first set $\mathcal A_{k_1}$, if its size is odd, use the same
path from the first vertex of $L_{k_1}$ to the last vertex of
$R_{k_1}$.  If its size is even, use instead the path supplied by
Lemma~\ref{lem:aligned} from the first vertex of $R_{k_1}$ to the last
vertex of $R_{k_1}$.  In either case, the local path has an endpoint at
the last vertex of $R_{k_1}$, which is the endpoint needed to connect
it to the next block.

Similarly, for the last set $\mathcal A_{k_m}$, if its size is odd,
use the path from the first vertex of $L_{k_m}$ to the last vertex of
$R_{k_m}$.  If its size is even, use the path from the first vertex of
$L_{k_m}$ to the last vertex of $L_{k_m}$.  In either case, it has an
endpoint at the first vertex of $L_{k_m}$, which is the endpoint needed
to connect it to the preceding block.

For $1\leq j<m$, join the local path for $\mathcal A_{k_j}$ to the
local path for $\mathcal A_{k_{j+1}}$ along the subpath of $Q$ from
the last vertex of $R_{k_j}$ to the first vertex of $L_{k_{j+1}}$.
Since $k_{j+1}\geq k_j+2$, the regions $L_{k_j}\cup R_{k_j}$ are
pairwise disjoint and occur in increasing order on $Q$.  Hence these
joining subpaths are also pairwise disjoint and meet the local paths
only at their endpoints.  The remaining pairs in all even-indexed
sets therefore lie on one simple path. The solid arcs in Figure~\ref{fig:distance-classes} schematically
represent the local paths associated with the even-indexed sets
$\mathcal A_k$.  These local paths are joined along the intervening
subpaths of $Q$ to form the first path.

Now let $h_1<\cdots<h_t$
be the odd indices for which $\mathcal A_{h_j}$ is nonempty.  Apply
exactly the same construction to these sets: set aside one pair from
each internal set of even size, use a path crossing from $L_{h_j}$ to
$R_{h_j}$ for every internal set, and use the appropriate one-sided
path when the first or last set has even size.  Since
$h_{j+1}\geq h_j+2$, the resulting local paths can again be joined
along disjoint subpaths of $Q$.  Thus the remaining pairs in all
odd-indexed sets lie on a second simple path. The corresponding construction for the odd-indexed sets is represented
by the dashed arcs in Figure~\ref{fig:distance-classes}; it gives the
second path.

It remains to include the removed pairs and the selected pairs in
$\mathcal B$. The dotted arcs in Figure~\ref{fig:distance-classes} indicate the two
collections assigned to the third path: the pairs set aside from the
internal blocks and the selected pairs in $\mathcal B$.  The subpaths
of $Q$ used below to join these pairs into one simple path are not
shown in the figure. A pair in $\mathcal A_k$ has the form
$\{kd+a,(k+1)d+a\}$.  Choose the removed pairs in increasing order of
$k$, and choose different values of $a$ when two indices $k$ are
consecutive.  This is possible because every set from which a pair is
removed has at least two pairs.  The removed pairs then form a matching.

Order the removed pairs by their indices $k$.  Use one removed pair, move
on $Q$ to the lower endpoint of the next one, and continue.  If the two
indices are consecutive, the connecting interval lies in their common
interval of length $d$.  Otherwise it moves forwards through the
intervals between them.  The distinct values of $a$ in the first case and
the order of the pairs show that no vertex is repeated.  Hence all removed
pairs lie on one simple path.

A set $\mathcal A_k$ from which a pair was removed has a nonempty set of
the same parity both before and after it.  It follows that all endpoints
of the removed pairs lie in
\begin{equation}\label{eq:middle-interval}
 \{d,d+1,\ldots,q-d-1\}.
\end{equation}
The selected pairs in $\mathcal B$ join the two intervals
$\{0,\ldots,d-1\}$ and $\{q-d,\ldots,q-1\}$ in the same order.
Lemma~\ref{lem:aligned} gives a path through them with an endpoint at
$d-1$ or at $q-d$.  The path through the removed pairs can be extended
to $d$ or to $q-d-1$.  Joining these two paths across the corresponding
edge of $Q$ gives one simple path.  If one set is empty, only the other
path is used.  Together with the two paths obtained from the two parities,
this proves the lemma.
\end{proof}

\begin{figure}[t]
  \centering
  \includegraphics[width=0.8\textwidth]
    {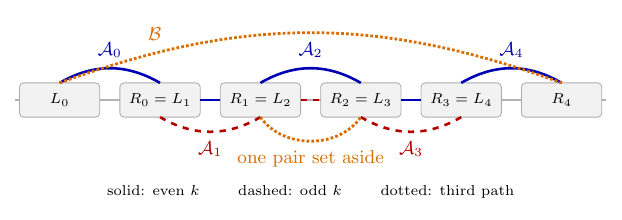}
  \caption{A schematic view of the decomposition used in the proof of
Lemma~\ref{lem:distance}.  Each solid or dashed arc represents a local
path supplied by Lemma~\ref{lem:aligned}, rather than a single edge.
The solid arcs correspond to the even-indexed sets $\mathcal A_k$, the
dashed arcs to the odd-indexed sets, and the dotted arcs to the pairs
assigned to the third path.  The subpaths of $Q$ used to concatenate
the local paths are not shown.}
\label{fig:distance-classes}
\end{figure}

\subsection{A separating family on a fixed path}\label{sec:one-graph}

Let $x_0,\ldots,x_{r-1}$ be vertices in this order on a fixed path $Q$.
The intervals of $Q$ between consecutive $x_i$ may contain other
vertices.  The proofs below use only the order of the $x_i$.

List the nonisolated vertices of $H$ as
$x_0,\ldots,x_{r-1}$ in their order on $Q$, and put
\[
 m=
 \begin{cases}
  r,& r\text{ odd},\\
  r+1,& r\text{ even}.
 \end{cases}
\]
Identify $x_i$ with $i\in\mathbb Z_m$; when $r$ is even, the label
$r$ is unused.  For an edge $\{i,j\}\in E(H)$, define
\[
 \sigma(i,j)=i+j\pmod m,\qquad
 \delta(i,j)=\min\{|i-j|,m-|i-j|\}.
\]
For $s\in\mathbb Z_m$, the edges satisfying $\sigma(i,j)=s$ form the
sum class $s$.  For $1\leq t\leq(m-1)/2$, the edges satisfying
$\delta(i,j)=t$ form the distance class $t$.

\begin{lemma}\label{lem:fixed-path-separator}
Let $H$ be a graph on $r\geq2$ nonisolated vertices of $Q$, and suppose
that $E(H)\cap E(Q)=\varnothing$.  There is a family of paths in
$Q\cup H$ that strongly separates the edges of $H$ and has size at most
\begin{equation}\label{eq:fixed-path-separator}
 \begin{cases}
  (5r-3)/2,& r\text{ odd},\\
  5r/2,& r\text{ even}.
 \end{cases}
\end{equation}
Moreover, the family may be chosen so that the edges of $H$ contained
in each path belong to a single sum class or to a single distance
class.
\end{lemma}

\begin{proof}
First suppose that $r$ is odd, so that $m=r$.  Fix
$s\in\mathbb Z_m$.  If $\sigma(i,j)=s$, then the ordinary integer sum
$i+j$ is either $s$ or $s+m$.  For a fixed integer $u$, the edges
$\{i,j\}$ satisfying $i+j=u$ form a nested matching: as their smaller
endpoints increase, their larger endpoints decrease.  Thus the edges
with integer sum $s$ form a nested matching in the first part of $Q$,
whereas those with integer sum $s+m$ form a nested matching in the last
part.

Apply Lemma~\ref{lem:nested} to these two matchings, choosing from each
matching a path whose endpoint faces the interval between them.  Join
the two paths along that interval of $Q$.  If one matching is empty,
use only the path for the other matching.  Hence every sum class is
contained in one simple path, and all sum classes together require at
most $m=r$ paths.

There are $(m-1)/2=(r-1)/2$ distance classes.  By
Lemma~\ref{lem:distance}, each distance class is contained in at most
three paths.  The total number of paths is therefore at most
\begin{equation}\label{eq:odd-count}
 r+3\frac{r-1}{2}=\frac{5r-3}{2}.
\end{equation}

The pair $(\sigma,\delta)$ determines the unordered pair $\{i,j\}$.
Indeed, the congruences
$i+j\equiv\sigma\pmod m$ and
$i-j\equiv\pm\delta\pmod m$ determine $i$ and $j$ up to exchange,
since $m$ is odd.  Now consider two distinct edges $e$ and $f$.  If
they belong to different sum classes, the sum-class path containing
$e$ avoids $f$, and the one containing $f$ avoids $e$.  If they belong
to the same sum class, they have different distances.  A distance-class
path containing $e$ then avoids $f$, and one containing $f$ avoids
$e$.  This proves the odd case.

Now let $r=2d$ be even, so that $m=r+1=2d+1$.  Temporarily extend the
subpath of $Q$ from $x_0$ to $x_{r-1}$ by an auxiliary vertex $y$
adjacent to $x_{r-1}$, and label $y$ by $2d$.  No edge of $H$ is
incident with $y$.  Apply the preceding sum-class construction modulo
$m$, obtaining at most $m$ sum-class paths.  For every distance
$t<d$, Lemma~\ref{lem:distance} gives at most three distance-class
paths.

Since $y$ is an endpoint of the auxiliary path and is incident with no
edge of $H$, every constructed path containing $y$ has $y$ as an
endpoint.  Deleting $y$ therefore leaves a simple path in $Q\cup H$
containing the same edges of $H$.  Hence the sum-class paths and the
paths for distances less than $d$ remain valid after $y$ is deleted.

It remains to treat distance $d$.  After $y$ is deleted, the graph
formed by all possible pairs of distance $d$ among the original
vertices is the path
\begin{equation}\label{eq:distance-d-path}
 d-1,2d-1,d-2,2d-2,\ldots,0,d.
\end{equation}
Its two alternating matchings are
\begin{equation}\label{eq:distance-d-matchings}
 \{\{i,i+d\}:0\leq i<d\},
 \qquad
 \{\{i,i+d+1\}:0\leq i<d-1\}.
\end{equation}
By Lemma~\ref{lem:aligned}, the selected edges of $H$ in each nonempty
matching are contained in one simple path.  Thus the distance-$d$ class
requires at most two paths.  Altogether, the number of paths is at most
\begin{equation}\label{eq:even-count}
 m+3(d-1)+2
 =(r+1)+3(d-1)+2
 =\frac{5r}{2}.
\end{equation}

Since $m=2d+1$ is odd, the pair $(\sigma,\delta)$ again determines the
unordered pair $\{i,j\}$.  The same argument as in the odd case proves
strong separation.  By construction, the edges of $H$ contained in
each path belong to a single sum class or to a single distance class.
Finally, the construction uses only the order of the labeled vertices,
so it remains valid when the intervals between them contain other
vertices of $Q$.
\end{proof}

\subsection{Completion of the proof}\label{sec:main-proof}

We use the following two known results.

\begin{theorem}[Fernandes, Hoppen, Kontogeorgiou, Mota, and Peng
\cite{FHKMP}]\label{thm:degenerate-separation}
Every $2$-degenerate graph $G$ admits a strongly separating family of
at most $|V(G)|$ paths.
\end{theorem}

Moreover, we may choose this family so that every edge is covered. Indeed, for each connected component with at least three vertices, the stronger form proved in~\ref{thm:degenerate-separation} places every edge in exactly two paths; a component consisting of a single edge is covered by that edge itself.
 
\begin{theorem}[Fan \cite{Fan}]\label{thm:fan-path-cover}
Every connected graph on \(r\) vertices can have its edges covered by at most $\lceil r/2\rceil$ paths.
\end{theorem}
We next describe the form of P\'osa rotation that we use.  Let
$Q=v_0v_1\cdots v_\ell$ be a path, and fix its endpoint $v_\ell$.  If
$v_0v_i$ is an edge for some $2\leq i\leq\ell$, replace the edge
$v_{i-1}v_i$ by $v_0v_i$.  This gives the new path
$v_{i-1}v_{i-2}\cdots v_0v_iv_{i+1}\cdots v_\ell$.  It has the same
vertex set as $Q$ and still has endpoint $v_\ell$.  Repeating this
operation gives further paths with the same vertex set and the same fixed
endpoint.

For a graph $R$ and $X\subseteq V(R)$, let $N_R(X)$ be the set of
vertices outside $X$ that have a neighbor in $X$.  Let $S$ be the set of
all endpoints different from $v_\ell$ that can be obtained from $Q$ by a
sequence of the rotations above.

We include the standard proof of the following lemma for the convenience
of the reader.

\begin{lemma}\label{lem:posa}
If $Q$ is a longest path of $R$, then
\begin{equation}\label{eq:posa}
 N_R(S)\subseteq N_Q(S)
 \qquad\text{and}\qquad
 |N_R(S)|\leq2|S|-1.
\end{equation}
\end{lemma}

\begin{proof}
Every path obtained by rotations has vertex set $V(Q)$.  Since $Q$ is a
longest path, none of its possible endpoints has a neighbor outside
$V(Q)$.

Let $x\in S$ and let $y\in N_R(x)\setminus S$.  Choose a path from $x$
to $v_\ell$ obtained from $Q$ by rotations.  If $xy$ is not an edge of
this path, perform one more rotation using $xy$.  If $xy$ is an edge of
the original path $Q$, then $y\in N_Q(S)$ immediately.  Otherwise, the
sequence of rotations contains a step at which an edge incident with $y$
is removed: this happens either when $xy$ was first added or in the last
rotation just described.  Consider the first such step.  The removed edge
belongs to the original path $Q$.  Whenever a rotation removes an edge,
one of its endpoints becomes a possible endpoint of a path with fixed
endpoint $v_\ell$, and hence belongs to $S$.  Since $y\notin S$, the
other endpoint of the removed edge belongs to $S$.  Thus
$y\in N_Q(S)$, and $N_R(S)\subseteq N_Q(S)$ follows.

The vertex $v_0$ belongs to $S$ and has degree one in $Q$.  Every other
vertex of $S$ has degree at most two in $Q$.  Therefore
$|N_Q(S)|\leq2|S|-1$, which proves the second assertion.
\end{proof}

\begin{proof}[Proof of Theorem~\ref{thm:main}]
Set $G_0=G$. Suppose that $G_{i-1}$ has an edge. Choose a longest path
$Q_i$ in $G_{i-1}$, fix one of its endpoints, and let $S_i$ be the set
given by Lemma~\ref{lem:posa}. Define $G_i=G_{i-1}-S_i$. Let $H_i$
be the graph formed by the edges of $G_{i-1}$ that have at least one
endpoint in $S_i$. Thus the edge sets $E(H_i)$ and $E(G_i)$ are
disjoint and together form $E(G_{i-1})$. Continue until the remaining
graph has no edges.

Put $s_i=|S_i|$, $h_i=|V(H_i)|$ and
$V(H_i)=S_i\cup N_{G_{i-1}}(S_i)$. Lemma~\ref{lem:posa} gives
\begin{equation}\label{eq:layer-count}
h_i=s_i+|N_{G_{i-1}}(S_i)|\leq3s_i-1.
\end{equation}
The sets $S_i$ are disjoint, while the sets $E(H_i)$ partition $E(G)$.
In particular,
\begin{equation}\label{eq:deleted-vertices}
\sum_i s_i\leq n.
\end{equation}

Let $B_i=E(Q_i)\cap E(H_i)$, and let $J_i$ be the graph with edge set
$E(H_i)\setminus B_i$. Finally, let $B$ be the graph with edge set
$\bigcup_i B_i$. For each $i$, every edge of $H_i$ lies in exactly one
of $B_i$ and $J_i$: it lies in $B_i$ if it is an edge of $Q_i$, and it
lies in $J_i$ otherwise. Since the edge sets $E(H_i)$ are disjoint and
together contain all edges of $G$, every edge of $G$ lies either in $B$
or in exactly one graph $J_i$, but not in both.

\medskip
{\noindent We {\bf Claim} that $B$ is $2$-degenerate.}

\begin{proof}[Proof of the Claim]

We define an ordering of the vertices for this purpose. This ordering
concerns only $B$, not the graphs $J_i$. Order the vertices by their
deletion steps. Thus every vertex of $S_i$ comes before every vertex
deleted at a later step. The order inside each set $S_i$ is arbitrary.
Put the vertices that remain after the process at the end, again in any
order.

Let $v\in S_i$, and let $vw$ be an edge of $B$ such that $w$ comes after
$v$. The edge $vw$ belongs to some $B_j$. If $j<i$, then one endpoint
of $vw$ belongs to $S_j$. This endpoint cannot be $v$, so $w\in S_j$
and $w$ comes before $v$, a contradiction. If $j>i$, then
$v\notin V(G_{j-1})$, so $vw\notin E(Q_j)$, again a contradiction.
Therefore $vw\in B_i$. Since $B_i\subseteq E(Q_i)$, the vertex $v$ has
at most two later neighbors in $B$.

Every vertex left at the end is isolated in the final remaining graph.
Thus all its neighbors in $B$ were deleted earlier and come before it.
Hence it has no later neighbor in $B$. Therefore $B$ is $2$-degenerate.
\end{proof}

By Theorem~\ref{thm:degenerate-separation}, $B$ admits a strongly
separating family $\mathcal B$ such that
\begin{equation}\label{eq:B-family}
|\mathcal B|\leq |V(B)|\leq n.
\end{equation}
As noted after Theorem~\ref{thm:degenerate-separation}, we may choose
$\mathcal B$ so that it also covers every edge of $B$.

Consider a component \(C\) of \(J_i\) that contains an edge, and let
\(r=|V(C)|\). Since \(C\) is nontrivial, every vertex of \(C\) is
incident with an edge of \(J_i\), and hence with an edge of \(H_i\).
Lemma~\ref{lem:posa} therefore implies that
$V(C)\subseteq V(Q_i)$.
Moreover, by the definition of \(J_i\), all edges of \(H_i\) that lie
on \(Q_i\) were removed, and hence
$$
 E(C)\cap E(Q_i)=\varnothing.
$$

Thus the vertices of \(C\) may be ordered according to their order
along \(Q_i\), and Lemma~\ref{lem:fixed-path-separator} applies. Let
\(\mathcal S_C\) be the resulting strongly separating family for the
edges of \(C\).

We also need a family that covers every edge of \(C\). This is useful
later when separating an edge of \(C\) from an edge belonging to a
different component: a path contained in \(C\) and containing the
former edge automatically avoids every edge of the other component.
Accordingly, by Theorem~\ref{thm:fan-path-cover}, there is a family
\(\mathcal C_C\) of at most \(\lceil r/2\rceil\) paths in \(C\) such
that every edge of \(C\) lies on some member of \(\mathcal C_C\).

The two families together have size at most
\begin{equation}\label{eq:component-count}
\begin{cases}
(5r-3)/2+(r+1)/2=3r-1,& r\text{ odd},\\
5r/2+r/2=3r,& r\text{ even}.
\end{cases}
\end{equation}
In particular,
$$
 |\mathcal S_C|+|\mathcal C_C|\le 3r.
$$

Let \(C_1,\ldots,C_m\) be the nontrivial components of \(J_i\), and put
\(r_k=|V(C_k)|\). By \eqref{eq:component-count}, the number of paths
used for \(C_k\) is at most \(3r_k\). Since the components are
vertex-disjoint and \(V(C_k)\subseteq V(H_i)\) for every \(k\), we have
$$
 \sum_{k=1}^m r_k\le |V(H_i)|=h_i.
$$

Hence the total number of paths used for all nontrivial components of
\(J_i\) is at most
$$
 3\sum_{k=1}^m r_k\le 3h_i.
$$

Let $\mathcal P$ be the union of $\mathcal B$ and all families
$\mathcal S_C$ and $\mathcal C_C$. We verify that $\mathcal P$ strongly
separates $E(G)$. Two edges of $B$ are separated by $\mathcal B$. Two
edges in the same component $C$ are separated by $\mathcal S_C$. If the
edges lie in different components, a path in $\mathcal C_C$ containing
the first edge avoids the second, and the corresponding path for the
other component gives the reverse separation. Finally, suppose that one
edge lies in $B$ and the other lies in a component $C$. A path in
$\mathcal B$ containing the first edge avoids $E(C)$, while a path in
$\mathcal C_C$ containing the second edge avoids $E(B)$. These cases
cover every pair of distinct edges.

Let $t$ be the number of steps. Since $E(G)\ne\varnothing$, we have
$t\geq1$. By~\eqref{eq:layer-count},
\eqref{eq:deleted-vertices}, and~\eqref{eq:B-family},
\begin{equation}\label{eq:final-count}
|\mathcal P|\leq n+3\sum_i h_i
\leq n+9\sum_i s_i-3t
\leq10n-3.
\end{equation}
This proves the theorem.
\end{proof}

\begin{remark}\label{rem:log-refinement}
A more careful count gives the sharper bound
\begin{equation}\label{eq:log-refinement}
\ssp(G)\leq\lfloor 10n-3\log_2 n-9\rfloor.
\end{equation}
Suppose that the process stops after $t\geq1$ deletion steps. Thus
$G_{t-1}$ has an edge, and deleting $S_t$ leaves the edgeless graph
$G_t$. Let $r=|V(G_t)|$. Thus $V(G_t)$ is precisely the set of
vertices left when the whole process ends.

For $0\leq i\leq t$, let $r_i=|V(G_i)|$. Hence $r_{i-1}$ is the
number of vertices before step $i$, while $r_i$ is the number after
that step. In particular, $r_0=n$, $r_i=r_{i-1}-s_i$, and $r_t=r$.
The fixed endpoint of $Q_i$ does not lie in $S_i$, so it survives
step $i$. Therefore $r_i\geq1$ for $1\leq i\leq t$, and in
particular $r\geq1$.

For $1\leq i\leq t$, the proof gives $h_i\leq3s_i-1$. We also have
$h_i\leq r_{i-1}$, since every vertex counted by $h_i$ belongs to
$G_{i-1}$. Set $\varepsilon_i=3s_i-h_i$. Thus
$h_i=3s_i-\varepsilon_i$ and $\varepsilon_i\geq1$. The quantity
$\varepsilon_i$ measures the saving at step $i$ over the bound
$h_i\leq3s_i$. We claim that
\begin{equation}\label{eq:one-step-saving}
\varepsilon_i\geq\log_2\frac{r_{i-1}}{r_i}.
\end{equation}

If $s_i\leq r_{i-1}/2$, then
$r_i=r_{i-1}-s_i\geq r_{i-1}/2$. Thus the logarithm in
\eqref{eq:one-step-saving} is at most $1$, and the claim follows from
$\varepsilon_i\geq1$.

Suppose that $s_i>r_{i-1}/2$. The bound $h_i\leq r_{i-1}$ gives
$\varepsilon_i\geq3s_i-r_{i-1}\geq s_i$. Since $r_i\geq1$, we also
have
$$
 \frac{r_{i-1}}{r_i}
 =1+\frac{s_i}{r_i}
 \leq1+s_i
 \leq2^{s_i}.
$$

Here $s_i$ is a positive integer because $S_i$ contains an endpoint
of $Q_i$. Hence the logarithm in \eqref{eq:one-step-saving} is at most
$s_i$. This proves the claim in the second case.

Summing \eqref{eq:one-step-saving} over \(1\leq i\leq t\), the
intermediate logarithmic terms cancel, leaving only the first and last
terms. Since $r_0=n$ and $r_t=r$, we obtain
$$
 \sum_{i=1}^t\varepsilon_i\geq\log_2\frac nr.
$$

We also have
$$
 \sum_{i=1}^t s_i=n-r,
$$

because exactly the vertices in $S_1,\ldots,S_t$ are deleted.

The family constructed in the proof satisfies
$$
 |\mathcal P|\leq n+3\sum_{i=1}^t h_i.
$$

Using the preceding relations in this count gives
\begin{equation}
\begin{aligned}
|\mathcal P|
&\leq n+3\sum_{i=1}^t(3s_i-\varepsilon_i)\
&=10n-9r-3\sum_{i=1}^t\varepsilon_i\
&\leq10n-9r-3\log_2\frac nr.
\end{aligned}
\end{equation}
The last expression equals
$10n-3\log_2 n-(9r-3\log_2 r)$. Since $r$ is a positive integer,
$\log_2 r\leq r-1$, and hence
$9r-3\log_2 r\geq6r+3\geq9$. Therefore
$|\mathcal P|\leq10n-3\log_2 n-9$. Since $|\mathcal P|$ is an
integer, this proves \eqref{eq:log-refinement}.
\end{remark}

\bibliographystyle{abbrv}
\bibliography{separation}

\end{document}